\documentclass{amsart} 
\author{Andy Hammerlindl}
\title{Polynomial Global Product Structure}

\usepackage{amsmath}
\usepackage{amssymb}
\usepackage{amsfonts}
\usepackage{amsthm}
\usepackage{graphicx}

\newcommand{\FF}{\mathcal{F}}
\newcommand{\GG}{\mathcal{G}}
\newcommand{\WW}{\mathcal{W}}
\newcommand{\Es}{E^s}
\newcommand{\Ec}{E^c}
\newcommand{\Eu}{E^u}

\newcommand{\Ws}{W^s}
\newcommand{\Wc}{W^c}
\newcommand{\Wu}{W^u}
\newcommand{\Wcu}{W^{cu}}
\newcommand{\Wcs}{W^{cs}}

\newcommand{\volume}{\operatorname{volume}}

\newcommand{\gps}{Global Product Structure}
\newcommand{\pgps}{polynomial \gps}
\newcommand{\lgps}{linear \gps}

\newtheorem{thm}{Theorem}[section]
\newtheorem{cor}[thm]{Corollary}
\newtheorem{lemma}[thm]{Lemma}
\newtheorem{prop}[thm]{Proposition}

\theoremstyle{remark}

\newtheorem*{defn} {\bf Definition}

\begin{document}

\begin{abstract}
    An Anosov diffeomorphism is topologically conjugate to an infranilmanifold
    automorphism if and only if it has \pgps.
\end{abstract}
\maketitle

\section{Introduction} 

A famous conjecture states that every Anosov diffeomorphism is topologically
conjugate to an algebraic example, specifically an infranilmanifold
automorphism.
In classification results toward this end \cite{brin1981anosov}
\cite{Manning}, a key step is establishing \gps.

Foliations $\FF$ and $\GG$ have \emph{\gps} if every leaf $\FF(x)$ intersects
every leaf $\GG(y)$ in a unique point $[x,y]$.
An Anosov diffeomorphism has \gps\ if its stable and unstable foliations have
\gps\ on the universal cover.
While it is not clear if \gps\ alone is enough to give a classification, the
following stronger condition is.

Foliations $\FF$ and $\GG$ have \emph{\pgps} if there is a polynomial $p$ such
that $d_\FF(x,[x,y]) + d_\GG(y,[x,y]) < p(d(x,y))$
for all $x$ and $y$.
Here, $d_\FF$ and $d_\GG$ are distances measured along the leaves, and $d$ is
distance on the ambient manifold.

\begin{thm} \label{polyiff}
    An Anosov diffeomorphism is topologically conjugate to an infranilmanifold
    automorphism if and only if it has \pgps.
\end{thm}
In particular, all known examples of Anosov systems have \pgps.
All of the ingredients needed to prove the forward direction appear in the
work of Brin and Manning \cite{brin1978fundamental}\cite{brin1981anosov}, and
the converse direction is a routine exercise in Lie group theory.  However,
the result seems to have gone unnoticed to date.

We give a specific application of the result.
A foliation $\FF$ is \emph{quasi-isometric} if, after lifting to the universal
cover, there is $Q > 1$ such that
$d_\FF(x,y)  \le  Q d(x,y) + Q$
for all $x$ and $y$ on the same lifted leaf.

\begin{thm} \label{qianosov}
    An Anosov diffeomorphism with quasi-isometric stable and unstable
    foliations is topologically conjugate to an infranilmanifold
    automorphism.
\end{thm}
The same analysis applies to partially hyperbolic systems and we give
several results, including the following.

\begin{thm} \label{qithree}
    For a (strongly) partially hyperbolic diffeomorphism on a 3-manifold $M$,
    the stable, center, and unstable foliations exist and are quasi-isometric
    if and only if
    $M$ is finitely covered by the 3-torus.
\end{thm}
\section{Anosov systems} \label{ansec} 

\begin{defn}
    Consider a diffeomorphism $f$ on a compact Riemannian manifold $M$
    with a $Tf$-invariant splitting
    $TM = \Es \oplus \Ec \oplus \Eu$
    and constants $\sigma < 1 < \mu$ such that
    \[
        \|Tf v^s\| < \sigma < \|Tf v^c\| < \mu < \|Tf v^u\|
    \]
    for all unit vectors $v^s \in \Es$, $v^c \in \Ec$, and $v^u \in \Eu$.

    \begin{itemize}
        \item
        If $\Ec$ is the zero bundle, then $f$ is called \emph{Anosov}.
        \item
        If exactly one of $\Es$ and $\Eu$ is the zero bundle, \\
        then $f$ is called \emph{weakly partially hyperbolic}.
        \item
        If all three subbundles are non-zero, \\
        then $f$ is called \emph{strongly partially hyperbolic}.
    \end{itemize}  \end{defn}
We first give two polynomial estimates which hold for all Anosov systems.

\begin{defn}
    For an ordered pair of foliations $(\FF,\GG)$, an \emph{$R$-rectangle} is a
    continuous map $\phi:[0,1]^2 \to M$ such each $\phi(\cdot, t)$ lies in a leaf
    of $\FF$, each $\phi(s, \cdot)$ lies in a leaf of $\GG$, and
    \[
        \sup \bigl\{ d_\GG( \phi(s, t_1), \phi(s, t_2) ) : s,t_1,t_2 \in [0,1] \bigr\} < R.
    \]
    The pair $(\FF,\GG)$ has \emph{polynomial bounds on rectangles} if there is a
    polynomial $p$ such that
    \[
        d_\FF(\phi(0,0),\phi(1,0)) < \frac{1}{p(R)}  \quad \Rightarrow \quad 
        d_\FF(\phi(0,1),\phi(1,1)) < 1
    \]
    for every $R$-rectangle $\phi$.
    We simply write that $\FF$ and $\GG$ have polynomial bounds on rectangles,
    if both $(\FF,\GG)$ and $(\GG,\FF)$ satisfy the property.
\end{defn}
The following observation is at the core of the proof in \cite{Brin-nw}.

\begin{lemma}
    For an Anosov diffeomorphism, the stable and unstable foliations, lifted
    to the universal cover, have
    polynomial bounds on rectangles.
          \end{lemma}
\begin{defn}
    A Riemannian manifold $\tilde M$ has \emph{polynomial growth of volume}, if
    there is a polynomial $p$ such that
    $ \volume ( B(x, R) ) < p(R) $
    for all $x \in \tilde M$ and $R > 0$.  Here $B(x, R)$ is the set of all
    points at a distance less than $R$ from $x$.

    \noindent A foliation $\FF$ has polynomial growth of volume if there is a
    polynomial $p$ such that
    $ \volume ( \FF(x, R) ) < p(R) $
    for all $x \in \tilde M$ and $R > 0$, and where $\FF(x, R)$ denotes all
    points reachable from $x$ by a path in $\FF(x)$ of length less than $R$.
    (That is, each leaf has polynomial growth, and the polynomial is
    independent of the leaf.)
\end{defn}
If a map $f:M \to M$ is uniformly expanding, the universal cover $\tilde M$ must
have polynomial growth of volume \cite{shub1970expanding}.  The same applies to
the uniform contraction or expansion inside a stable or unstable foliation.

\begin{lemma}
    In an Anosov (or partially hyperbolic) diffeomorphism, the stable and
    unstable foliations have polynomial growth of volume.
\end{lemma}
Brin showed that for an Anosov diffeomorphism with ``pinched'' spectrum, the
supporting manifold has a universal cover with polynomial growth of volume
\cite{brin1978fundamental}.  The proof relies on the following.

\begin{prop} \label{grow}
    Suppose $\FF$ and $\GG$ are $C^0$ foliations with $C^1$ leaves on a compact
    Riemannian manifold $M$, such that $T \FF$ and $T \GG$ are continuous as
    subbundles of $TM$.  
    If the lifted foliations on the universal cover $\tilde M$
    \begin{itemize}
        \item
        have \pgps,

        \item
        have polynomial bounds on rectangles, and

        \item
        have polynomial growth of volume,
    \end{itemize}
    then $\tilde M$ has polynomial growth of volume.
\end{prop}
Pinching is needed only to establish the first of the above three conditions.

\begin{cor} \label{corpoly}
    If an Anosov diffeomorphism $f:M \to M$ has \pgps, the universal cover has
    polynomial growth of volume.
\end{cor}
For a compact manifold $M$, the universal cover has polynomial growth of
volume if and only if the fundamental group does.  In such a case, a famous
result of Gromov shows that the fundamental group is virtually nilpotent
\cite{gromov1981groups}.  This fact allowed Brin and Manning to classify Anosov
diffeomorphisms with pinched spectra \cite{brin1981anosov}.  In fact, their
paper proves the following more general result.

\begin{thm}
    [Brin-Manning]
    If an Anosov diffeomorphism $f$ has \gps, and the universal cover has
    polynomial growth of volume, then $f$ is topologically conjugate to an
    infranilmanifold automorphism.
\end{thm}
Taken with Corollary \ref{corpoly}, this proves one direction of Theorem
\ref{polyiff}.

The other direction is an enjoyable exercise in basic Lie group theory.
%
As a hint, we suggest the following.

\begin{lemma}
    Let $G$ be a nilpotent group with a subset $V$ such that $[V,V] \subset
    V$.
    Any element of the form
    $ g = s_1 u_1 \cdots s_n u_n $
    with $s_i,u_i \in V$ may be rewritten as 
    \[        g = s_1 \cdots s_n \, a_1 \cdots a_\ell \, u_1 \cdots u_n  \]
    where $a_k \in [V,V]$ and where $\ell$ depends polynomially on $n$.

\end{lemma}
\begin{defn}
    We define \emph{\lgps} as \pgps\ where the associated polynomial may be
    taken as linear, that is, of the form $p(t) = a t + b$.
\end{defn}
The following is proved, but not explicitly stated in \cite{ham-conseq}.

\begin{prop} \label{qil}
    An Anosov diffeomorphism with quasi-isometric stable and unstable
    foliations has \lgps.
\end{prop}
Taken with Theorem \ref{polyiff}, this proves Theorem \ref{qianosov}.

\section{Partially hyperbolic systems} \label{phsec} 

First, consider weakly partially hyperbolic systems.  For simplicity, assume
$\Es$ is the zero bundle.  That is, the splitting is $TM = \Ec \oplus \Eu$.
For every weakly partially hyperbolic system, there is a unique foliation
tangent to the unstable subbundle $\Eu$.  For the center subbundle
$\Ec$, there may or may not be a tangent foliation.  It is not known if a
such a center foliation must be unique \cite{BW-dc}.

\begin{defn}
    A weakly partially hyperbolic system has \emph{(linear/polynomial) \gps},
    if there are foliations $\Wc$ and $\Wu$ tangent to $\Ec$ and $\Eu$ which
    satisfy the property after lifting to the universal cover.
\end{defn}
\begin{lemma} \label{weakrec}
    If a weakly partially hyperbolic diffeomorphism has \lgps, then $\Wc$ and
    $\Wu$ have polynomial bounds on rectangles.
\end{lemma}
\begin{proof}
    For the ordered pair $(\Wc,\Wu)$ the proof is as in the Anosov case.  For
    $(\Wu,\Wc)$, consider an $R$-rectangle $\phi$ with $d_u(\phi(0,1),\phi(1,1)) >
    1$.  We wish to bound $d_u(\phi(0,0),\phi(0,1))$ from below.
    By a compactness argument, there are constants $\gamma, \mu, \beta$ such
    that
    \[
        \|Tf v^c\| < \gamma < \mu < \|Tf v^u\| < \beta
    \]
    for all unit vectors $v^c \in \Ec$ and $v^u \in \Eu$.  Then,
    $
        d_u(f^n \phi(0,1), f^n \phi(1,1)) > \mu^n
    $
    for all $n > 0$.
    By \lgps, there is $n$ proportional to $\log R$ such that
    $
        d(f^n \phi(0,0), f^n \phi(1,1)) > 2 \gamma^n R.
    $
    Then
    \begin{align*}
        d( f^n \phi(1,0), f^n \phi(1,1) ) < \gamma^n R & \quad \Rightarrow \quad 
        d(f^n \phi(0,0), f^n \phi(1,0) ) > \gamma^n R \\ & \quad \Rightarrow \quad 
        d(\phi(0,0), \phi(1,0)) > \beta^{-n} \gamma^n R.
    \end{align*}
    As $n$ is proportional to $\log R$, this gives a polynomial bound.
    
\end{proof}
As in the Anosov case, the strong expansion in the unstable direction implies
that $\Wu$ has polynomial growth of volume, but no such claim can be made for
the center.  Lemma \ref{weakrec} and Proposition \ref{grow} give the following.

\begin{prop}
    If $f:M \to M$ is weakly partially hyperbolic with \lgps, and $\Wc$ has
    polynomial growth of volume, then the universal cover $\tilde M$ has
    polynomial growth of volume.
\end{prop}
The proof of Proposition \ref{qil} in \cite{ham-conseq} also applies to weakly
partially hyperbolic systems, giving the following.

\begin{cor}
    If $f:M \to M$ is weakly partially hyperbolic, $\Wc$ and $\Wu$ exist and are
    quasi-isometric, and $\Wc$ has
    polynomial growth of volume, then the universal cover $\tilde M$ has
    polynomial growth of volume.
\end{cor}
Now consider strongly partially hyperbolic systems.
Due to the three-way splitting $\Es \oplus \Ec \oplus \Eu$, we need to define a
notion of product structure inside a foliation.

\begin{defn}
    Suppose $\FF$ and $\GG$ are subfoliations of a foliation $\WW$.  Then $\FF$
    and $\GG$ have \emph{(linear/polynomial) \gps\ inside} $\WW$ if they have
    this property when restricted to each leaf of $\WW$.  Any choice of
    polynomial must be independent of the leaf.
\end{defn}
Most of the above results can be adapted to the case of product structure
inside a foliation, with almost no change to the proofs.

\begin{defn}
    Suppose a strongly partially hyperbolic system $f$ has
    foliations $\Wu$, $\Ws$, $\Wc$, $\Wcu$ and $\Wcs$ tangent to $\Eu$, $\Es$,
    $\Ec$, $\Ec \oplus \Eu$ and $\Ec \oplus \Es$.
    Then, $f$ has \emph{(linear/polynomial) \gps} if the following pairs
    satisfy the property on the universal cover{:}
    \begin{itemize}
        \item $\Wcs$ and $\Wu$,
        \item $\Ws$ and $\Wcu$,
        \item $\Ws$ and $\Wc$ inside $\Wcs$, and
        \item $\Wc$ and $\Wu$ inside $\Wcu$.
    \end{itemize}  \end{defn}
The following is proved, but not explicitly stated in \cite{ham-conseq}.

\begin{prop}
    A strongly partially hyperbolic diffeomorphism with quasi-isometric
    stable, center, and unstable foliations has \lgps.
\end{prop}
The next two results may be proved from their weakly partially hyperbolic
counterparts.

\begin{prop}
    If $f:M \to M$ is strongly partially hyperbolic with \lgps, and $\Wc$ has
    polynomial growth of volume, then the universal cover $\tilde M$ has
    polynomial growth of volume.
\end{prop}
\begin{cor} \label{phbound}
    If $f:M \to M$ is strongly partially hyperbolic, $\Ws$, $\Wc$ and $\Wu$ exist
    and are quasi-isometric, and $\Wc$ has polynomial growth of volume, then
    the universal cover $\tilde M$ has polynomial growth of volume.
\end{cor}
Note that one-dimensional foliations always have polynomial growth of volume.
One-dimensional center foliations are by far the most commonly studied case.

We now prove the last claim of the introduction.

\begin{proof}
    [Proof of Theorem \ref{qithree}.]
    By classical Bieberbach theory, any diffeomorphism on a finite quotient of
    the 3-torus lifts to a diffeomorphism on the 3-torus
    \cite{ha2002classification}.
    For partial hyperbolicity on the 3-torus, it is already known that the
    invariant foliations exist and are quasi-isometric
    \cite{BBI2}\cite{ham-thesis}.

    For the other direction, apply Corollary \ref{phbound} to see that the
    fundamental group has polynomial growth.
    Parwani showed that in this case, $M$ is finitely covered by a circle
    bundle over a torus \cite{parwani}.  For partially hyperbolic systems on
    such manifolds, the center foliations have been completely classified
    \cite{ham-thesis}\cite{ham-nil}.  The center foliation is quasi-isometric if
    and only if the circle bundle is the 3-torus.
\end{proof}
%


\medskip 

{\sc\small
IMPA,
Estrada Dona Castorina 110,
Rio de Janeiro / Brasil\ \  22460-320 }

\bibliographystyle{plain}
\bibliography{dynamics}

\end{document}